\documentclass[12pt]{amsart}

\usepackage{amsmath, amsfonts, amssymb, latexsym, amsthm}
\usepackage{graphicx}
\usepackage{palatino}
\usepackage[top=2.5cm, bottom=2.0cm, left=2.0cm, right=2.0cm,ignoreall]{geometry}
\usepackage[T1]{fontenc}
\usepackage{setspace}
\usepackage{url}
\usepackage{comment}

\usepackage{wrapfig}

\usepackage{tikz}

\usetikzlibrary{calc}
\usetikzlibrary{positioning}
\usetikzlibrary{decorations.markings}
\usetikzlibrary{decorations.pathreplacing}
\usetikzlibrary{patterns}

\theoremstyle{definition}
\newtheorem{theorem}{Theorem}

\newtheorem{proposition}[theorem]{Proposition}

\newtheorem{example}[theorem]{Example}
\newtheorem{question}{Question}

\allowdisplaybreaks

\title{Adversarial graph burning densities}

\address[Karen Gunderson and Hritik Punj]{Department of Mathematics, University of Manitoba, Winnipeg MB, Canada, R3T 2N2}
\address[JD Nir]{Department of Mathematics and Statistics, Oakland University, Rochester MI, US, 48326}
\address[William Kellough]{Department of Mathematics and Statistics, Memorial University, St John's NL, A1C 5S7}
\author{Karen Gunderson}
\email[Karen Gunderson]{karen.gunderson@umanitoba.ca}
\author{William Kellough}
\email[William Kellough]{wskellough@mun.ca}
\author{JD Nir}
\email[JD Nir]{jdnir@oakland.edu}
\author{Hritik Punj}
\email[Hritik Punj]{punjh@myumanitoba.ca}

\thanks{The first author gratefully acknowledges funding from NSERC, the second author was supported in part from the University of Manitoba Faculty of Science Undergraduate Research Award program, and the fourth author was supported by the University of Manitoba Undergraduate Research Award program.}


\date{\today}

\begin{document}

\begin{abstract}
Graph burning is a discrete-time process that models the spread of influence in a network. Vertices are either \emph{burning} or \emph{unburned}, and in each round, a burning vertex causes all of its neighbours to become burning before a new \emph{fire source} is chosen to become burning. We introduce a variation of this process that incorporates an adversarial game played on a nested, growing sequence of graphs. Two players, Arsonist and Builder, play in turns: Builder adds a certain number of new unburned vertices and edges incident to these to create a larger graph, then every vertex neighbouring a burning vertex becomes burning, and finally Arsonist `burns' a new fire source. This process repeats forever. Arsonist is said to win if the limiting fraction of burning vertices tends to 1, while Builder is said to win if this fraction is bounded away from 1.

The central question of this paper is determining if, given that Builder adds $f(n)$ vertices at turn $n$, either Arsonist or Builder has a winning strategy.  In the case that $f(n)$ is asymptotically polynomial, we give threshold results for which player has a winning strategy.
\end{abstract}

\maketitle

\section{Introduction}

Graph burning is a discrete-time process that models the spread of influence in a network.  Vertices are in one of two states: either \emph{burning} or \emph{unburned}.  In each round, a burning vertex causes all of its neighbours to become burning and a new \emph{fire source} is chosen: a previously unburned vertex whose state is changed to burning.  The updates repeat until all vertices are burning. The \emph{burning number} of a graph $G$, denoted $b(G)$, is the minimum number of rounds required to burn all of the vertices of $G$.  

Graph burning first appeared in print in a paper of Alon~\cite{nA92}, motivated by a question of Brandenburg and Scott at Intel, and was formulated as a transmission problem involving a set of processors. It was then independently studied by Bonato, Janssen, and Roshanbin~\cite{BJR14, BJR16, eR16} who gave bounds and characterized the burning number for various graph classes.  They showed in \cite{BJR16} that the burning number of every connected graph is equal to the burning number of one of its spanning trees and conjectured that if $G$ is a connected graph on $n$ vertices, then  $b(G) \leq \lceil n^{1/2} \rceil$.  If true, this would be the best possible upper bound as the path with $n$ vertices, $P_n$, satisfies $b(P_n) = \lceil n^{1/2} \rceil$.  Currently, the tightest known bound that holds for all $n$ was given by Bastide, Bonamy, Bonato, Charbit, Kamali, Pierron, and Rabie~\cite{BBBCKPR22} who showed that $b(G) \leq \lceil \sqrt{4n/3} \rceil +1$.  Norin and Turcotte recently showed~\cite{NT22} that $b(G) \leq (1+o(1))\sqrt{n}$ and so the conjecture holds asymptotically.  Here, we require only the following bound from~\cite{BBJRR18}:
\begin{equation}\label{best known upper bound of b(G)}
b(G) \leq  \sqrt{2n}.
\end{equation}

Burning density, introduced by Bonato, Gunderson, and Shaw~\cite{BGS20}, considers the burning process when the graph itself is growing over time.  Consider a sequence of graphs:
\[
G_1 \subseteq G_2 \subseteq G_3 \subseteq \cdots
\]
such that if $i < j$, the subgraph of $G_j$ induced by the vertices $V(G_i)$ is $G_i$.  That is, the `old graph' does not change, but at each step `new vertices' arrive with some edges either between new vertices or between new vertices and old vertices.  The burning process in this setting evolves as: the graph grows, the fire spreads, and we select an additional fire source.

In this context, there may never be a step when every vertex is burning, but one can consider the proportion of burning vertices.  Given the sequences of graphs $\mathbf{G} = (G_i)_{i \geq 1}$ and fire sources $\mathbf{v} = (v_i)_{i \geq 1}$, let $B_n$ be the set of burning vertices at the end of round $n$ and $V_n = V(G_n)$ be the set of all vertices at the end of round $n$. Then we define the \emph{lower burning density}
$\underline{\delta}(\mathbf{G}, \mathbf{v}) = \liminf_{n \to \infty} \frac{|B_n|}{|V_n|}$, the \emph{upper burning density} $\overline{\delta}(\mathbf{G}, \mathbf{v}) = \limsup_{n \to \infty} \frac{|B_n|}{|V_n|}$, and the \emph{burning density} $\delta(\mathbf{G}, \mathbf{v}) = \lim_{n \to \infty} \frac{|B_n|}{|V_n|}$, if it exists.  Note that it is possible to construct examples where the fraction of burning vertices fluctuates infinitely often and the burning density does not exist.

In traditional graph burning there is a choice of where to place the fires, but the process is otherwise deterministic.  In this paper, we modify the setup for burning densities to define an adversarial `game'.  In this game there are two players, Builder and Arsonist, who play as follows: given a function $f: \mathbb{Z}^+ \to \mathbb{Z}_{\ge 0}$, known to both players, and beginning with an empty graph $G_0$, at time step $n \ge 1$, 
\begin{itemize}
    \item Builder adds $f(n)$ vertices to $G_{n-1}$ and edges incident to these new vertices to create a connected graph $G_n$,
	\item the fire spreads, and
	\item Arsonist selects an additional fire source $v_n$.
\end{itemize}
As before, we use $V_n$ to refer to the set of vertices of $G_n$ and $B_n$ to refer to the set of vertices of $G_n$ that are burning at the conclusion of step $n$ of the process.  Note that the definition of burning density does not require $G_n$ to be connected, but we impose this restriction in order to prevent Builder from playing the trivial strategy of adding $f(n)$ independent vertices. In this game (that goes on forever), Arsonist wins if the burning density is $\delta(\mathbf{G}, \mathbf{v}) =1$ and Builder wins if $\overline{\delta}(\mathbf{G}, \mathbf{v}) < 1$.  That is, Arsonist wins if they can eventually burn all but a negligible fraction of the vertices and Builder wins if they can eventually guarantee that for all time steps going forward, a constant fraction of the vertices are unburned. If $\overline{\delta}(\mathbf{G}, \mathbf{v}) = 1$ but $\underline{\delta}(\mathbf{G}, \mathbf{v}) < 1$, which is to say if the burning density does not exist but the proportion of burning vertices approaches $1$ infinitely often, we say that neither player wins and the game ends in a draw.  Assume that both Builder and Arsonist play `perfectly': among all possible choices, they always choose the best.

One may equivalently view this game as measuring the worst-case burning density of a graph burning process in which the graphs are known to be connected and the orders of the graphs to be burned are known, but the specific choice of $G_n$ is only revealed on turn $n$.  

The game we have described is an adversarial zero-sum game with perfect information. This means at most one player can win and both players know of all options available to each player at every step in the game. Such games have a long history in the literature, tracing back to von Neumann and Morgenstern~\cite{NM47}. Gale and Stewart~\cite{GS53} introduced the idea of infinite games in which a winner is selected only after countably many turns. Infinite games are used in mathematical logic and have been used in computer science to study computation processes~\cite{yG93}. Other examples of games on graphs that have (or may potentially have) infinitely many rounds include a Maker-Breaker game for infinite connected components of the origin in a subgraph of the infinite grid~\cite{DF-R21}, a Ramsey-game where a win for one of the players involves an infinite graph~\cite{HKNPRS17}, a variation of Cops and Robbers on an infinite graph~\cite{fL16}, and energy-parity games~\cite{CD12}.

The central question of this paper is determining whether, given that Builder adds $f(n)$ vertices at turn $n$, Arsonist or Builder has a winning strategy. Our main result proves there is a threshold at which functions which grow polynomially switch from being Arsonist-win to Builder-win. Our results concern the asymptotic behaviour of the function controlling the number of vertices Builder receives on turn $n$. We direct the reader to~\cite{CLRS09} for definitions of asymptotic notation.

\begin{theorem}\label{thm:simple main}
Let $\alpha > 0$  and $f: \mathbb{Z}^+ \to \mathbb{Z}_{\ge 0}$ satisfy $f(n) = \theta(n^\alpha)$. If $\alpha < 1$, then Arsonist has a winning strategy in the adversarial burning game in which Builder adds $f(n)$ vertices on turn $n$. If $\alpha \ge 1$, then Builder has a winning strategy.
\end{theorem}

Indeed, Theorem~\ref{thm:simple main} is a direct consequence of a broader classification theorem which considers functions that may fluctuate between growth rates.

\begin{theorem}\label{thm:main}
Let $f: \mathbb{Z}^+ \to \mathbb{Z}_{\ge 0}$. If, in the adversarial graph burning game, Builder adds $f(n)$ vertices on turn $n$, then:
\begin{itemize}
\item Arsonist has a winning strategy if there exists $\alpha < 1$ such that $f(n) = O(n^\alpha)$ and $\omega(n^{2\alpha-1})$, and
\item Builder has a winning strategy if $f(n) = \Omega(n)$.
\end{itemize}
\end{theorem}

In~\cite{BJR16}, it was shown that determining extremal values of burning numbers for connected graphs can always be reduced to the restricted problem on trees.  Similarly, we show that, in terms of understanding which player has a winning strategy, it suffices to assume that Builder always builds trees.

Given a function $f(n)$, define a \emph{legal construction} for Builder to be a sequence of graphs $(G_1, G_2, \dots)$ such that for each $n\in \mathbb{Z}^+$, $G_n$ is a connected graph with $\sum_{k=1}^n f(k)$ vertices, $G_n \subseteq G_{n+1}$, and the subgraph of $G_{n+1}$ induced by $V(G_n)$ is $G_n$.

\begin{theorem} \label{thm:why_trees}
Let $f(n)$ be given. For every legal construction $\mathbf{G}=(G_1, G_2, \dots)$ for Builder and for all choices of fire sources, $\mathbf{v}$, every legal construction $\mathbf{T}=(T_1, T_2, \dots)$ such that $T_n$ is a spanning tree of $G_n$, for all $n\in \mathbb{Z}^+$, satisfies $\underline{\delta}(\mathbf{T}, \mathbf{v}) \leq \underline{\delta}(\mathbf{G}, \mathbf{v})$ and $\overline{\delta}(\mathbf{T}, \mathbf{v}) \leq \overline{\delta}(\mathbf{G}, \mathbf{v})$.
\end{theorem}

In particular, if Arsonist has a strategy that succeeds given any sequence of trees Builder selects, they have a winning strategy for all legal constructions. Similarly, if Builder has a winning strategy, they also have a winning strategy using a legal construction consisting of a sequence of trees.

We prove Theorem~\ref{thm:why_trees} in Section~\ref{sec:why_trees} and Theorem~\ref{thm:main} in Section~\ref{sec:poly} before closing with a few unresolved questions regarding the adversarial burning game.

\subsection{Acknowledgements}
The authors offer thanks to Ruiyang Chen who gave useful feedback on some preliminary discussions of the problems considered in this paper as well as the anonymous referee who offered several useful refinements to the results and examples in this paper.

\section{Restricting Builder to trees} \label{sec:why_trees}

In this section we prove Theorem~\ref{thm:why_trees}.

\begin{proof}
For a legal construction $\mathbf{G}=(G_1, G_2, \dots)$, let $\mathbf{T}=(T_1, T_2, \dots)$ be any sequence of graphs such that $T_n \subseteq T_{n+1}$, the subgraph of $T_{n+1}$ induced by $V(T_n)$ is $T_n$, and $T_n$ is a spanning tree of $G_n$ for all $n\in \mathbb{Z}^+$. Let $\mathbf{v} = (v_1, v_2, \dots)$ be any sequence of fire sources for Arsonist in $\mathbf{G}$. Consider the game in which Arsonist burns the same sequence of vertices $\mathbf{v}$ when Builder uses the legal construction $\mathbf{T}$. 

Since $T_n$ contains a subset of the edges of $G_n$ for every $n\in \mathbb{Z}^+$, the number of burning vertices in $T_n$ is at most the number of burning vertices in $G_n$. Furthermore, $T_n$ and $G_n$ have the same number of vertices for all $n\in \mathbb{Z}^+$. Thus $\underline{\delta}(\mathbf{T}, \mathbf{v}) \leq \underline{\delta}(\mathbf{G}, \mathbf{v})$ and $\overline{\delta}(\mathbf{T}, \mathbf{v}) \leq \overline{\delta}(\mathbf{G}, \mathbf{v})$.

\end{proof}

Theorem~\ref{thm:why_trees} guarantees Builder has an optimal strategy in which the sequence of graphs they construct is a growing sequence of trees. As a consequence, we may restrict our consideration of Builder's strategies to those in which Builder only builds trees, though doing so will not always be necessary. 

\section{Polynomial thresholds} \label{sec:poly}

In this section we prove Theorem~\ref{thm:main} by describing winning strategies in each scenario. We start by proving that Arsonist has a winning strategy when the number of vertices Builder is given is bounded by two sublinear polynomials.

\begin{proposition}\label{Arsonist wins for alpha<1}
Let $\alpha < 1$ and let $f: \mathbb{Z}^+ \to \mathbb{Z}_{\ge 0}$. If $f(n) = O(n^{\alpha})$ and $f(n) = \omega(n^{2\alpha-1})$, then Arsonist has a winning strategy in the adversarial burning game when Builder receives $f(n)$ new vertices at step $n$.
\end{proposition}

\begin{proof}
  Let $c > 0$ and $N_1\in \mathbb{Z}^+$ be such that $f(n) \leq c n^\alpha$ for all $n \ge N_1$. At a given step $i$, use $G_i$ to denote the graph that Builder has constructed at step $i$ (based on previous steps and the choices of the Arsonist) and set $V_i = V(G_i)$.

We allow Arsonist to choose vertices to burn arbitrarily in the first $N_1$ steps of the game. Recursively define a sequence of integers $(N_k)_{k \geq 1}$ as follows.  For every $k\geq 1$ define  $A_k = b(G_{N_k})$, the burning number of $G_{N_k}$, and $N_{k+1} = N_k + A_k$.  Arsonist's strategy will be to burn the graph $G_{N_k}$ in $A_k$ steps (which we bound by Equation~\eqref{best known upper bound of b(G)}) between time $N_k$ and time $N_{k+1}$, ignoring any vertices that had been burning before time $N_k$. Thus we have $|B_{N_k}| \geq |V_{N_{k-1}}|$ for all $k\geq 2$ as the number of vertices burning at time $N_k$ is at least the number of vertices in $G_{N_{k-1}}$.

We first estimate the number of new vertices added between time $N_{k-1}$ and $N_k$: $\sum_{n = N_{k-1}+1}^{N_k} f(n)$.  For this, we use the inequality
\begin{equation}\label{binomial}
(1+x)^{1+\alpha} \leq 1+2(1+\alpha)x,
\end{equation}
which holds for $0\leq x<1$ when $-1 \le \alpha < 1$. For $k\geq 2$,
    \begin{align*}
        \sum_{n=N_{k-1} + 1}^{N_k} f(n) &\leq c\sum_{n=N_{k-1} + 1}^{N_k} n^\alpha\\
        &\leq c \int_{N_{k-1}}^{N_k+1} x^\alpha dx \\
        &= \frac{c}{\alpha + 1}\left((N_k+1)^{\alpha + 1} - (N_{k-1})^{\alpha + 1}\right) \\
        &= \frac{c}{\alpha + 1}\left( (N_{k-1} + A_{k-1}+1)^{\alpha + 1} - (N_{k-1})^{\alpha + 1}\right) \\
        &= \frac{c}{\alpha + 1} (N_{k-1})^{\alpha + 1}\left( \left(1 + \frac{A_{k-1}}{N_{k-1}} + \frac{1}{N_{k-1}}\right)^{\alpha + 1} - 1\right) \\
        &\leq \frac{c}{\alpha + 1} (N_{k-1})^{\alpha + 1}\left( \left(1 + \frac{2A_{k-1}}{N_{k-1}}\right)^{\alpha + 1} - 1\right) \\
        &\leq \frac{c}{\alpha + 1} (N_{k-1})^{\alpha + 1}\left( 1 + \frac{4(\alpha + 1)A_{k-1}}{N_{k-1}} - 1\right) & \text{(by Equation~\eqref{binomial})} \\
        &= 4c (N_{k-1})^{\alpha} A_{k-1}\\
        &\le 4c (N_{k-1})^{\alpha} \cdot \sqrt{2|V_{N_{k-1}}|}\,. & \text{(by Equation~\eqref{best known upper bound of b(G)})}\\
    \end{align*}

To estimate the total number of vertices, we consider two cases depending on whether $\alpha \geq 1/2$ or $\alpha < 1/2$.  First, assume that $\alpha \geq 1/2$ so that $2\alpha -1 \geq 0$.

We first claim that
\begin{equation}\label{eq:vg_lb}
    |V_{N_k}| = \omega((N_k)^{2\alpha}).
\end{equation}
Fix $M > 0$ and choose $N$ large enough such that $f(n) \ge 4\alpha M\cdot n^{2\alpha-1}$ for all $n \ge N$. Then for $k$ large enough that $N_k \ge N$,
\begin{align*} 
|V_{N_k}| &= \sum_{n=1}^{N_k}f(n)\\
    &\geq \sum_{n=N+1}^{N_k}4\alpha M n^{2\alpha-1}\\
    &\ge 4\alpha M \int_N^{N_k}x^{2\alpha-1} dx\\
    &\ge 2M \left(\left(N_k\right)^{2\alpha}-N^{2\alpha})\right)\\
    &\ge M\left(N_k\right)^{2\alpha}
\end{align*}
as $N_k \ge 2^{1/(2\alpha)}N$. As $M$ was arbitrary, we have shown $|V_{N_k}| = \omega((N_k)^{2\alpha})$.

To estimate the fraction of burning vertices in the graph $G_{N_k}$, note that
\[ |V_{N_k}| = |V_{N_{k-1}}| + \sum_{n=N_{k-1} + 1}^{N_k} f(n).\]

Therefore, using Equation~\eqref{eq:vg_lb},
\[ \frac{\sum_{n=N_{k-1} + 1}^{N_k} f(n)}{|V_{N_{k-1}}|} \le \frac{4\sqrt{2}c(N_{k-1})^\alpha}{\sqrt{|V_{N_{k-1}}|}} = 4\sqrt{2}c \frac{(N_{k-1})^\alpha}{\omega((N_{k-1})^\alpha)} = o(1).\]

Returning to the case where $\alpha < 1/2$, note that the condition on the lower bound for the function $f$ only requires that when $n$ is sufficiently large, $f(n) > 0$. Therefore for sufficiently large $k$ we may use the bound $|V_{N_{k-1}}| \geq N_{k-1}/2$ to show
\[ \frac{\sum_{n=N_{k-1} + 1}^{N_k} f(n)}{|V_{N_{k-1}}|} \le \frac{4\sqrt{2}c(N_{k-1})^\alpha}{\sqrt{|V_{N_{k-1}}|}} \leq 8c \frac{(N_{k-1})^\alpha}{N_{k-1}^{1/2}} = o(1).\]

Thus in either case,
\begin{align*}
\liminf_{k \to \infty} \frac{|V_{N_{k-1}}|}{|V_{N_k}|} &= \liminf_{k \to \infty} \frac{|V_{N_{k-1}}|}{|V_{N_{k-1}}| + \sum_{n=N_{k-1} + 1}^{N_k} f(n)}\\
    &= \liminf_{k\to \infty} \frac{|V_{N_{k-1}}|}{(1+o(1))|V_{N_{k-1}}|} \\
    &= 1
\end{align*}

For any $n \ge N_1$, choose $k$ largest such that $n \ge N_k$. Then $|B_n| \ge |B_{N_k}| \ge |V_{N_{k-1}}|$ and $|V_n| \le |V_{N_{k+1}}|$, so we have
\begin{align*}
\underline{\delta}(\mathbf{G},\mathbf{v}) &= \liminf_{n \to \infty} \frac{|B_n|}{|V_n|}\\
    &\ge \liminf_{k \to \infty} \frac{|V_{N_{k-1}}|}{|V_{N_{k+1}}|}\\
    &= \left(\liminf_{k \to \infty} \frac{|V_{N_{k-1}}|}{|V_{N_{k}}|}\right) \left(\liminf_{k \to \infty} \frac{|V_{N_{k}}|}{|V_{N_{k+1}}|}\right)\\
    &= 1,
\end{align*}
where we note that the limit infimum is bounded above by $1$ and therefore exists, which permits splitting the limit across the product $\frac{|V_{N_{k-1}}|}{|V_{N_k}|} \cdot \frac{|V_{N_k}|}{|V_{N_{k+1}}|}$. So as
\[ 1 \ge \overline{\delta}(\mathbf{G},\mathbf{v}) \ge \underline{\delta}(\mathbf{G},\mathbf{v}) = 1,\]
the burning density, $\delta(\mathbf{G}, \mathbf{v})$, exists and equals $1$, so Arsonist wins.

\end{proof}

Note that the lower bound conditions on the function $f$ in Proposition~\ref{Arsonist wins for alpha<1} are necessary to guarantee a winning strategy for the Arsonist.  To see this, consider the following examples.

\begin{example}\label{ex:alpha<1}
First, for $\alpha < 1/2$, since the lower bound condition on the function $f$ in Proposition~\ref{Arsonist wins for alpha<1} reduces to simply $f(n) > 0$ for large enough $n$, we consider a function $f$ that alternates between $0$ and $\lfloor n^{\alpha} \rfloor$.  Throughout, we consider a game in which Builder builds a path, always extending only one end.  Given any $N_0$ and $|V_{N_0}|$, let $f(n) = 0$ for $N_0 < n \leq N_1$, with $N_1$ large enough so that $|V_{N_1}| < N_1^{\alpha}$.  Then, for $N_2 = N_1 + \lfloor N_1^{\alpha/2} \rfloor$, define $f(n) = \lfloor n^{\alpha} \rfloor$ when $N_1 < n \leq N_2$.  The total number of vertices at time $N_2$ is
\[
|V_{N_2}| \geq |V_{N_1}| + N_1^{\alpha} \cdot \lfloor N_1^{\alpha/2} \rfloor \geq |V_{N_1}| + N_1^{\alpha} \cdot (N_1^{\alpha/2} - 1) \geq N_1^{3\alpha/2} - N_1^{\alpha}.
\]

To bound from above the number of burning vertices at time $N_2$, we assume that all vertices in $V_{N_1}$ are burning. As Builder constructs a path, the vertices Arsonist burns from time $N_1+1$ to $N_2$ form (possibly overlapping) subpaths of burned vertices containing $1, 3, \ldots, 2\ell-1$ vertices, where
\[ \ell = N_2-(N_1+1)+1 = \lfloor N_1^{\alpha/2} \rfloor,\]
with at most $N_1^{\alpha/2}$ additional vertices burning due to spread from $V_{N_1}$. Together, this gives
\[
|B_{N_2}| \leq |V_{N_1}| + N_1^{\alpha/2} + \sum_{\ell = 1}^{\lfloor N_1^{\alpha/2} \rfloor} (2\ell-1) \leq N_1^{\alpha}+ N_1^{\alpha/2} + (N_1^{\alpha/2})^2 \leq 3 N_1^{\alpha}.
\]
Thus, 
\[
\frac{|B_{N_2}|}{|V_{N_2}|} \leq \frac{3N_1^\alpha}{N_1^{3\alpha/2} - N_1^{\alpha}}.
\]
We can repeat these phases of growth, setting $f(n) = 0$ from $N_{2i}$ to $N_{2i+1}$, chosen so that $|V_{N_{2i+1}}| < N_{2i+1}^\alpha$, and $f(n) = \lfloor n^\alpha \rfloor$ from $N_{2i+1}$ to $N_{2i+2} = N_{2i+1} + \lfloor N_{2i+1}^{\alpha/2} \rfloor$, to get
\begin{equation}\label{eq:alpha<1/2 ub bd}
\frac{|B_{N_{2i+2}}|}{|V_{N_{2i+2}}|} \leq \frac{3N_{2i+1}^\alpha}{N_{2i+1}^{3\alpha/2} - N_{2i+1}^{\alpha}}.
\end{equation}
For $N_{2i+1}$ sufficiently large, the upper bound in equation~\eqref{eq:alpha<1/2 ub bd} is arbitrarily small.  Repeating these phases of growth shows that for any choice of fire sources, $\underline{\delta}(\mathbf{G}, \mathbf{v}) = 0$, so we do not have $\delta(\mathbf{G}, \mathbf{v}) = 1$ and Arsonist does not win. 
\end{example}

\begin{example}\label{ex:1/2<alpha<1}
    Consider now an example where, for $1/2 \leq \alpha < 1$, the function $f$ fluctuates between $f(n) = \lfloor n^{\alpha} \rfloor$ and $f(n) = \lfloor n^{2\alpha-1} \rfloor$.  Again, assume that Builder builds a path, always extending only one end.

    Fix $0 < \varepsilon < 1/8$.  Given any $N_0$, let $M_0$ be the number of vertices added by time $N_0$.  Assuming that $f(n) = \lfloor n^{2\alpha-1} \rfloor$ for $N_0 < n \leq N$, then as $N$ tends to infinity, the number of vertices at time $N$ is
    \begin{align*}
    |V_N| &= M_0 + \sum_{n = N_0+1}^N \lfloor n^{2\alpha-1} \rfloor\\
        &= M_0 + (1+o(1))\int_{N_0}^N x^{2\alpha-1}\ dx\\
        &=M_0 + \frac{(1+o(1))}{2\alpha}\left(N^{2\alpha} - N_0^{2\alpha} \right)\\
        &=\frac{(1+o(1))}{2\alpha} N^{2\alpha},
    \end{align*}
    for any fixed value of $N_0$.
    Thus, given $N_0$ and the number of vertices added up to this point, let $N_1$ be large enough so that if $f(n) = \lfloor n^{2\alpha-1} \rfloor$ for $N_0 < n \leq N_1$, then
    \[
    \left\vert |V_{N_1}| - \frac{N_1^{2\alpha}}{2\alpha} \right\vert \leq \varepsilon N_1^{2 \alpha}.
    \]
    
    Set $N_2 = N_1 + \lfloor N_1^{\alpha} \rfloor$ and for $N_1 < n \leq N_2$, let $f(n) = \lfloor n^{\alpha} \rfloor$.  To bound from above the number of burning vertices at time $N_2$, we can assume that all vertices in $V_{N_1}$ and all vertices added up to time $N_1+\lfloor N_1^{\alpha}/2 \rfloor$ are burning.  Accounting for the spread of burning vertices from these and the (possibly overlapping) subpaths surrounding the vertices Arsonist selected after time $N_1+\lfloor N_1^{\alpha}/2 \rfloor$, we see that
    \begin{align}
        |B_{N_2}| &\leq |V_{N_1}| + \sum_{n = N_1 +1}^{N_1 + \lfloor N_1^{\alpha}/2 \rfloor} n^{\alpha} + (N_2 - (N_1+\lfloor N_1^{\alpha}/2 \rfloor)) + \sum_{\ell = 1}^{\lfloor N_1^{\alpha} \rfloor - \lfloor N_1^{\alpha}/2 \rfloor} (2\ell-1) \notag\\
        &\leq \left(\frac{1}{2\alpha} + \varepsilon\right) N_1^{2\alpha} + \int_{N_1+1}^{N_1 + N_1^{\alpha}/2+1} x^{\alpha}\ dx + \frac{N_1^\alpha}{2}+1 + \left(\frac{N_1^{\alpha}}{2} + 1 \right)^2 \label{eq:burning-ub-ex}
    \end{align}
    To estimate the integral in equation~\eqref{eq:burning-ub-ex}, we extend the inequality in Equation~(\ref{binomial}) by noting
    \begin{equation}\label{extended binomial}
    (1+x)^{\alpha+1} \leq 1 + (\alpha +1) x + (\alpha+1)\alpha x^2,
    \end{equation}
    which holds for $0 \leq x< 1$ for any $0 \leq \alpha < 1$. Therefore,
    \begin{align}
    \int_{N_1+1}^{N_1 + N_1^\alpha/2 + 1} x^{\alpha}\ dx &=\frac{1}{\alpha+1}\left((N_1 + N_1^{\alpha}/2 + 1)^{\alpha+1}  - (N_1+1)^{\alpha+1}\right) \notag\\
        &= \frac{(N_1+1)^{\alpha+1}}{\alpha+1}\left(\left(1 + \frac{N_1^{\alpha}}{2(N_1+1)}\right)^{\alpha+1} - 1 \right) \notag \\
        &\leq \frac{(N_1+1)^{\alpha+1}}{\alpha+1} \left(1 + \frac{(\alpha + 1)N_1^{\alpha}}{2(N_1+1)} + \frac{(\alpha+1)\alpha N_1^{2\alpha}}{4(N_1+1)^2} -1\right) \notag & \text{(by Equation~(\ref{extended binomial}))}\\
        &= \frac{N_1^{\alpha} (N_1+1)^\alpha}{2} + \frac{\alpha N_1^{2\alpha}}{4(N_1+1)^{1-\alpha}} \notag\\
        &\leq \frac{N_1^{2\alpha}}{2} (1+1/N_1)^\alpha + \frac{\alpha N_1^{3 \alpha -1}}{4} \notag\\
        &\leq \frac{N_1^{2 \alpha}}{2} \left(1 + \frac{2\alpha}{N_1} \right) + \frac{\alpha N_1^{3 \alpha -1}}{4} \notag & \text{(by Equation~(\ref{binomial}))}\\        
        &\leq \frac{N_1^{2\alpha}}{2} + 2 \alpha N_1^{3\alpha-1} \label{eq:integral-est}
    \end{align}
    where we note we applied Equation~(\ref{binomial}) with $\alpha$ instead of $1+\alpha$, which changes the restriction on $\alpha$ from $-1 \le \alpha < 1$ to $0 \le \alpha < 2$. We assumed $1/2 \le \alpha < 1$ and so this formulation of Equation~(\ref{binomial}) applies.
    
    Since $\left(\frac{N_1^{\alpha}}{2} + 1\right)^2 = \frac{N_1^{2\alpha}}{4} + N_1^{\alpha} + 1 \leq  \frac{N_1^{2\alpha}}{4} + 2N_1^{\alpha}$, combining equations~\eqref{eq:burning-ub-ex} and \eqref{eq:integral-est} gives
    \begin{align*}
    |B_{N_2}| &\leq \left(\frac{3}{4} + \frac{1}{2\alpha} + \varepsilon \right)N_1^{2\alpha} + \frac{N_1^\alpha}{2}+1 + 2\alpha N_1^{3\alpha -1} + 2N_1^{\alpha}\\
        &\leq \left(\frac{3}{4} + \frac{1}{2\alpha} + \varepsilon \right)N_1^{2\alpha} + 5 N_1^{3\alpha -1}.
    \end{align*}
    Meanwhile,
    \begin{align*}
        |V_{N_2}| &\geq |V_{N_1}| + \sum_{n = N_1+1}^{N_1 + \lfloor N_1^\alpha \rfloor} (n^{\alpha} - 1) \\
        &\geq |V_{N_1}| + \lfloor N_1^{\alpha} \rfloor ((N_1+1)^{\alpha} - 1)\\
    &    \geq \left(\frac{1}{2 \alpha } - \varepsilon \right)N_1^{2\alpha} + (N_1^{\alpha} - 1)^2\\
    & = \left(1+\frac{1}{2 \alpha } - \varepsilon \right)N_1^{2\alpha} - 2N_1^{\alpha}+1.
    \end{align*}
    Thus, the fraction of burning vertices at time $N_2$ is bounded by
    \begin{equation}\label{eq:alpha>1/2 ub bd}
\frac{|B_{N_2}|}{|V_{N_2}|} \leq \frac{\left(\frac{3}{4} + \frac{1}{2\alpha} + \varepsilon \right)N_1^{2\alpha} + 5 N_1^{3\alpha -1}}{\left(1+\frac{1}{2 \alpha } - \varepsilon \right)N_1^{2\alpha} - 2N_1^{\alpha}+1}.
    \end{equation}
Repeating this two phase construction as we did in Example~\ref{ex:alpha<1} gives a function $f$ such that for any sequence of vertices Arsonist chooses to burn, the upper bound in equation~\eqref{eq:alpha>1/2 ub bd} approaches arbitrarily close to $\frac{3/4 + 1/(2\alpha) + \varepsilon}{1 + 1/(2\alpha) - \varepsilon}$, which is strictly less than $1$. Therefore $\underline{\delta}(\mathbf{G}, \mathbf{v}) < 1$ and so Arsonist does not win.
\end{example}

Next, we consider the adversarial burning game in which the number of vertices Builder is given grows either linearly or superlinearly. 

\begin{proposition}\label{Builder wins for alpha=1}
Let $f: \mathbb{Z}^+ \to \mathbb{Z}_{\ge 0}$. If $f = \Omega(n)$, then Builder has a winning strategy in the adversarial burning game when given $f(n)$ new vertices at step $n$. 
\end{proposition}

\begin{proof}
We consider the strategy in which Builder constructs an increasingly long path by always adding new vertices to one end of the previous path.

Let $n_0, \alpha > 0$ be such that for each $n \ge n_0$, $f(n) \ge \alpha n$. Consider the state of the game on some turn $N > n_0$. First, consider the case where $|V_N| < 2N^2$. Regardless of Arsonist's strategy, some initial portion of Builder's path consists entirely of burning vertices while fire sources that Arsonist has burned more recently form small (potentially overlapping) subpaths. To upper bound $|B_N|$, we choose some turn $k_0$ satisfying $N > k_0 \ge n_0$ and we assume that the parts of the path added in the process up to turn $k_0$ are entirely burning and that fires started after turn $k_0$ do not intersect. This gives us three collections of burning vertices: an initial segment of length $\sum_{k=1}^{k_0} f(k)$, an additional $N-k_0$ vertices extending that initial segment that are potentially burned from that segment, and $N-k_0$ non-intersecting subpaths containing $1,3,\ldots,2(N-k_0)-1$ vertices, giving the estimate
\[ |B_N| \le \sum_{k=1}^{k_0} f(k) + (N-k_0) + \sum_{\ell=1}^{N-k_0} (2\ell-1) = \sum_{k=1}^{k_0} f(k) + (N-k_0)^2 + O(N). \]
The total number of vertices by turn $N$ is
\[ |V_N| = \sum_{k=1}^N f(k) \ge \sum_{k=1}^{k_0} f(k) + \sum_{k=k_0+1}^N \alpha k \ge \sum_{k=1}^{k_0} f(k) + \frac{\alpha}{2}(N^2-k_0^2) \]
and so
\[ \sum_{k=1}^{k_0} f(k) \le |V_N| - \frac{\alpha}{2}(N^2-k_0^2).\]
Thus for any choice of $k_0$ we have a burning density of at most
\begin{align*}
\frac{|B_N|}{|V_N|} &\le \frac{\sum_{k=1}^{k_0} f(k) + (N-k_0)^2 + O(N)}{|V_N|}\\
&\leq \frac{|V_N| - \frac{\alpha}{2}(N^2-k_0^2) + (N-k_0)^2 + O(N)}{|V_N|}\\
&= 1 - \frac{\frac{\alpha}{2}(N^2-k_0^2) - (N-k_0)^2 + O(N)}{|V_N|}.
\end{align*}
We choose
\[ k_0 = \left \lfloor \frac{2}{2+\alpha} N\right \rfloor \]
which gives
\[ \frac{\alpha}{2}(N^2-k_0^2) - (N-k_0)^2 + O(N) = \frac{\alpha^2}{2(\alpha+2)}N^2 + O(N).\]
Thus for sufficiently large $N$ we have
\begin{align*}
    \frac{|B_N|}{|V_N|} &\le 1 - \frac{\frac{\alpha^2}{2(\alpha+2)}N^2 + O(N)}{|V_N|}\\
    &< 1 - \frac{\frac{\alpha^2}{2(\alpha+2)}N^2 + O(N)}{2N^2}\\
    &= 1 - \frac{\alpha^2}{4(\alpha+2)} + o(1)\\
    &< 1 - \frac{\alpha^2}{5(\alpha+2)}.
\end{align*}

Now if instead $|V_N| \ge 2N^2$, we can bound
\[ |B_N| \le \sum_{k=1}^N 2k-1 = N^2\]
and thus we have 
\[ \frac{|B_N|}{|V_N|} \leq \frac{N^2}{2N^2} = \frac{1}{2}.\]
Therefore
    \[\limsup_{N\to \infty}\frac{|B_N|}{|V_N|} \leq \max\left( \frac{1}{2}, 1 - \frac{\alpha^2}{5(\alpha+2)} \right) < 1\]
and so Builder wins.
\end{proof}

In the case where the function $f$ fluctuates between a linear function and a sublinear function, Builder can no longer necessarily win, as shown by the following example.

\begin{example}\label{ex:not always linear}
We construct a function for which there are $0 < \alpha < 1$ and $\beta > 0$ so that either $f(n) > \beta n$ or else $f(n) = \lfloor n^{\alpha} \rfloor$. For any fixed time $N_0$, letting $f(n) > \beta n$ until time $N_1$ gives
\[
|V_{N_1}| = \sum_{n=1}^{N_1} f(n) > \sum_{n=N_0}^{N_1} \beta n = \frac{\beta}{2}N_1^2 - O(N_1),
\]
so choose $N_1$ large enough that
\[
|V_{N_1}| > \frac{\beta}{4}N_1^2.
\]
Set $N_2 = N_1 + \lceil \sqrt{2|V_{N_1}|} \rceil$ so that, using Equation~\eqref{best known upper bound of b(G)}, by time $N_2$ Arsonist can burn all of the vertices in $V_{N_1}$ no matter which sequence of graphs Builder chose and which fire sources Arsonist chose before turn $N_1$.  For $N_1 < n \leq N_2$, let $f(n) = \lfloor n^{\alpha} \rfloor$.  Note that by the assumption on $N_1$, $N_1 < \sqrt{\frac{4}{\beta}|V_{N_1}|}$.  The number of new vertices added during this time is
\begin{align*}
 \sum_{n = N_1 + 1}^{N_2} f(n) &\leq\lceil \sqrt{2|V_{N_1}|} \rceil N_2^{\alpha}\\
 &= \lceil \sqrt{2|V_{N_1}|} \rceil \left(N_1 + \lceil \sqrt{2|V_{N_1}|} \rceil \right)^{\alpha}\\
    &\leq 2\sqrt{|V_{N_1}|} \left(\sqrt{\frac{4}{\beta}|V_{N_1}|} + 2\sqrt{|V_{N_1}|} \right)^{\alpha}\\
    &\leq 4(\beta^{-1/2} + 1)^{\alpha}|V_{N_1}|^{(\alpha+1)/2}.
\end{align*}
Thus, the fraction of burning vertices at time $N_2$ is
\[
\frac{|B_{N_2}|}{|V_{N_2}|} \geq \frac{|V_{N_1}|}{|V_{N_1}| + 4(\beta^{-1/2} + 1)^{\alpha}|V_{N_1}|^{(\alpha+1)/2}}.
\]
We repeat this process, letting $f(n) > \beta n$ between $N_{2i}$ and $N_{2i+1}$, choosing $N_{2i+1}$ so that $|V_{N_{2i+1}}| > \frac{\beta}{4}N_{2i+1}^2$ and, because $(\alpha+1)/2 < 1$,
\[ \lim_{i \to \infty} \frac{|V_{N_{2i+1}}|}{|V_{N_{2i+1}}| + 4(\beta^{-1/2} + 1)^{\alpha}|V_{N_{2i+1}}|^{(\alpha+1)/2}} \to 1, \]
and $f = \lfloor n^\alpha \rfloor$ between $N_{2i+1}$ and $N_{2i+2} = N_{2i+1} + \lceil \sqrt{2|V_{N_{2i+1}}|} \rceil$. Because there exists a sequence of vertices Arsonist can choose to burn such that $\overline{\delta}(\mathbf{G}, \mathbf{v}) = 1$, Builder cannot have a winning strategy.
\end{example}

\section{Further directions} \label{sec:further_directions}

While Theorem~\ref{thm:main} resolves the question of which player wins the adversarial burning game for a large collection of functions, there are clearly many functions not covered by either the results given here or a direct application of the proof techniques. By alternating between exponentially many turns of no new vertices and one turn with exponentially many vertices, one can readily construct examples that fluctuate between different growth rates in which neither player wins, and even where the lower burning density is 0 and the upper burning density is 1.

While a classification of those functions for which either Arsonist or Builder has a winning strategy seems to be out of reach, it is natural to seek to expand the class of functions for which the winner is known. One approach would be to determine whether the problem exhibits monotonicity under certain conditions. It seems natural that if Builder has a winning strategy given $f(n)$ vertices, any additional vertices could only help.

\begin{question}\label{Monotonicity for Builder}
Let $f,g: \mathbb{Z}^+ \to \mathbb{Z}_{\ge 0}$ be functions such that $f(n)\leq g(n)$ for all $n\in \mathbb{Z}^+$. If Builder wins the adversarial burning game when given $f(n)$ vertices at every step $n$, does it necessarily follow that Builder still wins if they are given $g(n)$ vertices at every step?
\end{question}

Perhaps surprisingly, a reformulation of Question~\ref{Monotonicity for Builder} from Arsonist's perspective does not hold: if Arsonist has a winning strategy when Builder receives $g(n)$ vertices, giving Builder fewer vertices may hurt Arsonist. As seen by Examples~\ref{ex:alpha<1} and \ref{ex:1/2<alpha<1}, for any positive $\alpha < 1$ there are functions $f(n)$ that fluctuate between $\lfloor n^\alpha \rfloor$ and smaller values such that although Arsonist can win when Builder is given $\lfloor n^\alpha \rfloor$ vertices on each turn, they can no longer win when Builder is only given $f(n)$ vertices. If Builder can actually win when given $f(n)$ vertices, we could negatively resolve Question~\ref{Monotonicity for Builder}, but because of the possibility that neither player wins, the question remains open. These examples do, however, motivate a weaker version of Question~\ref{Monotonicity for Builder} from Arsonist's perspective.

\begin{question}\label{Monotonicity for Arsonist}
Let $f,g: \mathbb{Z}^+ \to \mathbb{Z}_{\ge 0}$ be functions such that $f(n)\leq g(n)$ for all $n\in \mathbb{Z}^+$. If Arsonist wins the adversarial burning game when Builder is given $g(n)$ vertices at every step $n$, does it necessarily follow that Builder cannot win if they are given $f(n)$ vertices at every step $n$?
\end{question}

One could also consider a weaker type of monotonicity. Suppose instead of decreasing the number of vertices Builder is given each turn, we only require that at each step the total number of vertices Builder has received is smaller. That is, instead of requiring $f(n)\leq g(n)$ for each $n\in \mathbb{Z}^+$, we have the weaker condition $\sum_{k=1}^n f(k) \leq \sum_{k=1}^n g(k)$.

In each instance we analyzed a specific strategy for Builder, we considered the consequences of Builder constructing a path. The conjecture that paths have extremal burning number makes this a natural choice of strategy. However, it is not uncommon in combinatorial games that an extremal example is not optimal in the context of a game. This suggests the following question.

\begin{question}
Let $f:\mathbb{Z}^+ \to \mathbb{Z}_{\ge 0}$. If Builder has some winning strategy in the adversarial burning game adding $f(n)$ vertices at step $n$, is building a path guaranteed to be a winning strategy for Builder?
\end{question}

If true, one would no longer need to consider which strategy Builder chooses, essentially reducing the question of determining a winner to bounding the burning density of a path that increases in length by $f(n)$ vertices at step $n$.

Finally, one may also consider variations of the adversarial burning game where either Builder or Arsonist is replaced by a random process. Mitsche, Pra\l{}at, and Roshanbin \cite{MPR17} studied the graph burning process when new fire sources were selected uniformly at random, which generalizes naturally to the adversarial game, replacing Arsonist by a uniform random process. Alternatively, one could replace Builder with a \emph{uniform random recursive tree}, which is generated by adding vertices one by one, selecting an existing vertex uniformly at random to be the parent of each new vertex. For more on recursive trees, see \cite{Dr09, MS95}.

\end{document}